\documentclass[12pt,reqno,letter]{amsart}

\usepackage[all]{xy}
\usepackage{amsmath}
\usepackage{amssymb,stmaryrd} 
\usepackage{amsthm}
\usepackage{texdraw}

\newtheoremstyle{break-italic}% name
  {0.8ex}%      Space above, empty = `usual value'
  {}%      Space below
  {\itshape}% Body font
  {0cm}%         Indent amount (empty = no indent, \parindent = para indent)
  {\bfseries}% Thm head font
  {}%        Punctuation after thm head
  {\newline }% Space after thm head: \newline = linebreak
  {}%         Thm head 
\newtheoremstyle{break-roman}% name
  {0.8ex}%      Space above, empty = `usual value'
  {}%      Space below
  {\normalfont}% Body font
  {}%         Indent amount (empty = no indent, \parindent = para indent)
  {\bfseries}% Thm head font
  {}%        Punctuation after thm head
  {\newline }% Space after thm head: \newline = linebreak
  {}%         Thm head spec

\theoremstyle{break-italic} 
\newtheorem{theorem}{Theorem}
\newtheorem{lemma}[theorem]{Lemma}
\newtheorem{corollary}[theorem]{Corollary}
\newtheorem{proposition}[theorem]{Proposition}

\theoremstyle{break-roman} 
\newtheorem{definition}[theorem]{Definition}
\newtheorem{example}[theorem]{Example}
\newtheorem{remark}[theorem]{Remark}

\newcommand{\kb}{{\mathcal B}}
\newcommand{\kd}{{\mathcal D}}
\newcommand{\ke}{{\mathcal E}}
\newcommand{\kj}{{\mathcal J}}
\newcommand{\kk}{{\mathcal K}}
\newcommand{\kl}{{\mathcal L}}
\newcommand{\ko}{{\mathcal O}}

\DeclareMathOperator{\expdim}{expdim}
\DeclareMathOperator{\length}{length}
\DeclareMathOperator{\im}{Im}
\DeclareMathOperator{\jet}{jet}
\DeclareMathOperator{\mult}{mult}

\DeclareMathAlphabet{\mathds}{U}{dsrom}{m}{n}

\newcommand{\C}{{\mathds C}}

\newcommand{\F}{{\mathds F}}
\newcommand{\pp}{{\mathds P}}

\newcommand{\tom}[1]{}

\newcommand{\intmult}{\cdot}

\SelectTips{cm}{12}

\begin{document}

   \parindent0cm

   \title[Triple Point]{Triple-Point Defective Surfaces}
   
   \author{Luca Chiantini}
      \address{Universit\'a degli Studi di Siena\\
     Dipartimento di Scienze Matematiche e Informatiche\\
     Pian dei Mantellini, 44
     I -- 53100 Siena
     }
   \email{chiantini@unisi.it}
  \urladdr{http://www.dsmi.unisi.it/newsito/docente.php?id=4}
   
   \author{Thomas Markwig}
   \address{Universit\"at Kaiserslautern\\
     Fachbereich Mathematik\\
     Erwin-Schr\"odinger-Stra\ss e\\
     D -- 67663 Kaiserslautern
     }
   \email{keilen@mathematik.uni-kl.de}
   \urladdr{http://www.mathematik.uni-kl.de/\textasciitilde keilen}
   \thanks{The second author was supported by the EAGER node of Torino,
     and by the Institute for Mathematics and its Applications (IMA),
     University of Minnesota.}


   \date{30th July, 2009.}

   \begin{abstract}
     In this paper we study the linear series $|L-3p|$ of hyperplane
     sections with a triple point $p$ on a surface $S$ embedded via a very
     ample line bundle $L$ for a \emph{general} point $p$. If this linear
     series does not have the expected dimension we call $(S,L)$
     \emph{triple-point defective}. We show that on a triple-point defective
     surface through a general point every hyperplane section
     has either a triple component or the surface is rationally ruled
     and the hyperplane section contains twice a fibre of the ruling.
   \end{abstract}

   \maketitle

   \section{Introduction}\label{sec:tpd}

   Throughout this note, $S$ will be a smooth projective surface,
   $K=K_S$ will denote the canonical class and $L$ will be a
   divisor class on $S$ such that $L$ is \emph{very ample} and $L-K$
   is \emph{ample and base-point-free}.

   The classical \emph{interpolation problem} for the pair $(S,L)$ is devoted
   to the study of the varieties:
   \begin{displaymath}
     V^{gen}_{m_1,\dots,m_n}=\big\{C\in |L|\;\big|\; p_1,\ldots,p_n\in
     S\mbox{ general},\;\mult_{p_i}(C)
     \geq m_i\big\}.
   \end{displaymath}

   In a more precise formulation, we start from the incidence variety:
   \begin{displaymath}
     \kl_{m_1,\dots,m_n}=\{(C,(p_1,\ldots,p_n))\in|L|\times
     S^n\;|\;\mult_{p_i}(C)\geq m_i\}
   \end{displaymath}
   together with the canonical projections:
   \begin{equation}\label{eq:alphabeta}
     \xymatrix{
       \kl_{m_1,\dots,m_n}\ar[r]^\alpha\ar[d]_\beta & S^n\\
       |L|=\pp(H^0(L)^*)
     }
   \end{equation}
   As for the map $\alpha$, the fibre  over a fixed point
   $(p_1,\dots,p_n)\in S^n$
   is just the linear series $|L-m_1p_1-\dots-m_np_n|$ of effective
   divisors in
   $|L|$ having a point of multiplicity at least $m_i$ at $p_i$.
   These fibres being irreducible, we deduce that if $\alpha$ is
   \emph{dominant} then $\kl_{m_1,\dots,m_n}$ has a unique
   irreducible component, say $\kl_{m_1,\dots,m_n}^{gen}$, which dominates
   $S^n$.
   The closure of its image
   \begin{displaymath}\label{VVV}
     V_{m_1,\dots,m_n}:=V_{m_1,\dots,m_n}(S,L):=\overline{\beta(\kl_{m_1,\dots,m_n}^{gen})}
   \end{displaymath}
   under $\beta$ is an irreducible closed subvariety of
   $|L|$, a \emph{Severi variety} of $(S,L)$.

   Imposing a point of multiplicity $m_i$ corresponds to killing
   $\binom{m_i+1}2$ partial derivatives, so that
   \begin{displaymath}
     \dim|L-m_1p_1-\dots-m_np_n|\geq
     \max\left\{-1,\dim|L|-\sum_{i=1}^n\binom{m_i+1}2\right\},
   \end{displaymath}
   and one expects that the previous inequality is in fact an equality,
   for the choice
   of general points $p_1,\dots,p_n\in S$.

   When this is not the case, then 
   the pair $(S,L)$, is called \emph{defective}
   and is endowed with some special structure. By abuse of notation we
   sometimes call the surface $S$ defective, if $L$ is understood.
   \smallskip

   The case when $m_i=2$ for all $i$ has been classically considered (and
   solved)  by Terracini, who classified in \cite{Ter22} double--point 
   defective surfaces.
   If $S$ is double-point defective, then a general curve
   $C\in|L-2p_1-\dots-2p_n|$ has a double component passing through each
   point $p_i$,
   and the map $\beta$ in Diagram \eqref{eq:alphabeta} has positive dimensional
   fibres.

   When the multiplicities grow, the situation becomes much more complicated.
   Even in the case $S=\pp^2$, the situation is not understood
   and there are several, still unproved conjectures on the structure
   of defective embeddings (see \cite{Cil01} for an introductory survey).

   Let us point out a first difference between the case of multiplicity two
   and the case of higher multiplicity. It is easy to show  that imposing on
   $|L|$ multiplicity two at {\it one} general point always  yields three
   independent conditions, so that $\dim|L-2p|=\dim |L|-3$, the
   expected dimension.  Contrary to this, on some surfaces, it turns out that
   even imposing
   {\it just one} point of multiplicity $3$, one may obtain
   a defective behaviour.

   \begin{example}\label{ex:hirzebruch}
     Let $S=\F_e\stackrel{\pi}{\longrightarrow}\pp^1$ be a
     Hirzebruch surface, $e\geq 0$. We denote by
     $F$ a fibre of $\pi$ and by $C_0$ the section of $\pi$ of
     minimal self intersection $C_0^2=-e$ -- both of which are smooth
     rational curves.
     The general element $C_1$ in the linear system
     $|C_0+eF|$ will be a section of $\pi$ which does not
     meet $C_0$ (see e.g.\ \cite{Har77}, Theorem 2.17).

     Consider now the divisor $L=(2+b)\cdot F+C_1=(2+b+e)\cdot
     F+C_0$ for some fixed $b\geq 0$. Then for a
     general $p\in S$ there are curves $D_p\in|bF+C_1-p|$ and there is a
     unique curve $F_p\in |F-p|$, in particular $p\in F_p\cap D_p$. For each
     choice of $D_p$ we have
     \begin{displaymath}
       2F_p+D_p \in|L-3p|.
     \end{displaymath}
     Since $F{\intmult}L=1=F{\intmult}(L-F)$ we see that every curve in $|L-3p|$ must
     contain $F_p$ as a double component, i.e.
     \begin{displaymath}
       |L-3p|=2F_p+|bF+C_1-p|.
     \end{displaymath}
     Moreover, since $bF+C_1$ is base-point
     free (see \cite{Har77}, Theorem 2.17) we have (see \cite{FP05},
     Lemma 35)
     \begin{multline*}
       \dim|bF+C_1-p|=\dim|bF+C_1|-1\\=h^0\big(\pp^1,\ko_{\pp^1}(b+e)\big)+
       h^0\big(\pp^1,\ko_{\pp^1}(b)\big)-2=2b+e
     \end{multline*}
     and, using the notation from above,
     \begin{displaymath}
       \dim(V_3)\geq \dim|bF+C_1-p|+2=2b+e+2.
     \end{displaymath}
     However,
     \begin{displaymath}
       \dim|L|= h^0\big(\pp^1,\ko_{\pp^1}(2+b+e)\big)
       +h^0\big(\pp^1,\ko_{\pp^1}(2+b)\big)-1
       =2b+e+5,
     \end{displaymath}
     and thus
     \begin{displaymath}
       \expdim(V_3)=\dim|L|-4=2b+e+1<2b+e+2=\dim(V_3).
     \end{displaymath}
     We say, $(\F_e,L)$ is \emph{triple-point defective}, see
     Definition \ref{def:tpd}.

     Note, moreover, that $L$ is very ample, as is $L-K_S$ for
     $b\geq\max\{0,e-3\}$ (see \cite{Har77}, Corollary 2.18), and that
     \begin{displaymath}
       (L-K)^2=\big((4+2e+b)\cdot F+3\cdot C_0)^22=24+6b+3e>16.
     \end{displaymath}
     \hfill$\Box$
   \end{example}
   \medskip

   It is interesting to observe that, even though, in the previous example,
   the general element of $|L-3p|$ is non-reduced, still the map $\beta$
   of Diagram \eqref{eq:alphabeta} has finite general fibres, since
   the general element of $|L-3p|$ has no triple components.
   \smallskip

   The aim of this note is to investigate the structure of pairs $(S,L)$
   for which the linear system $|L-3p|$ for $p\in S$ general has dimension
   bigger that the expected value $\dim|L|-6$. We will show that, under
   some assumptions, ruled surfaces
   are the only case of triple point defective surfaces.

   \begin{definition}\label{def:tpd}
     We say that the pair $(S,L)$ is \emph{triple-point defective}  or, in
     classical
     notation, that \emph{$(S,L)$ satisfies one Laplace equation} if
     \begin{displaymath}
       \dim|L-3p|>\max\{-1,\dim|L|-6\}=\expdim|L-3p|
     \end{displaymath}
     for $p\in S$ general.
   \end{definition}

   \begin{remark}\label{L3}
     Going back to Diagram \eqref{VVV}, one sees that $(S,L)$ is
     triple-point defective if and only if either:
     \begin{itemize}
     \item $\dim|L|\leq 5$ and the projection
       $\alpha:\kl_3\rightarrow S$ dominates, or \smallskip
     \item $\dim|L|>5$ and the general fibre of the map $\alpha$ has
       dimension at least $\dim|L|-5$.
     \end{itemize}
     In particular, $(S,L)$ is triple-point defective if and only if the
     map $\alpha$ is
     \emph{dominant} and
     \begin{displaymath}
       \dim(\kl_3^{gen})>\dim|L|-4.
     \end{displaymath}
   \end{remark}

   The particular case in which the general fibre of the map $\beta$ in
   Diagram \eqref{VVV} is positive-dimensional,
   (i.e.  the general member of $V_3$ contains a triple
   component through $p$) has been studied in \cite{Cas22},
   \cite{FI01}, and \cite{BC05}. We will recall the classification of such
   surfaces
   in Theorem \ref{thm:notfinite} below. Notice that these surfaces are
   almost
   always singular (i.e. $L$ is not very ample), so that they do not appear
   in the statement of  our main theorem, {\it where, indeed, we make no
     assumptions on the
     dimension of the fibres of $\beta$}.
   \smallskip

   One of the major subjects
   in algebraic interpolation theory, namely Segre's conjecture on
   defective linear systems \emph{in the plane},
   suggests in our situation that, when $(S,L)$ is triple-point defective,
   then the general element of $|L-3p|$ must be non-reduced, with a double
   component
   through $p$.
   \smallskip

   We will show here, under some assumptions, that this extension of
   Segre's conjecture for triple point defectivity holds for a single triple point.

   Our main result is:

   \begin{theorem}\label{thm:aim1}
     Let $L$ be a very ample line bundle on $S$, such that $L-K$
     is ample and base-point-free. Assume $(L-K)^2>16$ and $(S,L)$ is triple-point
     defective.

     Then $S$ is ruled in the embedding defined by $L$.
     Moreover a general curve $C\in|L-3p|$ contains the fibre of the
     ruling through $p$
     as fixed component with multiplicity at least two.
   \end{theorem}

   \begin{remark}
     In the paper \cite{CM07a} we classify all triple-point
     defective linear systems $L$ on ruled surfaces satisfying the
     assumptions of Theorem \ref{thm:aim1}, and it follows from this
     classification that the linear system $|L-3p|$ will contain the
     fibre of the ruling  through $p$ precisely with multiplicity two
     as a fixed component. In particular, the map $\beta$ will
     automatically be generically finite.
   \end{remark}

   Our method is an application of Reider's analysis of rank $2$
   bundles arising from triple points which do not impose independent
   conditions. Under the assumption that $(L-K)^2>16$, the bundle is
   Bogomolov unstable, and we show that from the destabilising divisors
   $A$ and $B=L-K-A$ one gets the multiple fibre. We point out that we
   obtain in this way a natural geometric construction for the
   non--reduced divisor which must be part of any defective linear system.
   
   This application of Reider's construction for the investigation of
   defective surfaces was introduced by Beltrametti, Francia and
   Sommese in \cite{BFS89}. We will refer to \cite{BFS89}
   for the first main properties of the destabilising divisors $A$ and $B$
   (see Section 4 below).
   
   Then, we will use the assumption``$L-K$ ample and base-point-free''
   to control curves of low degree
   on $S$. The freeness of $L-K$ seems unavoidable in the argument of
   the crucial Lemma 14, which in turn implies that $B$ is fixed part free
   and determines a ruling on $S$.

   Let us finish by pointing out in the following corollary what
   happens if we apply our result to 
   $\pp^2$ and its blow ups, and notice
   that, combining results in \cite{Xu95}
   and \cite{Laz97} Corollary\ 2.6, one can give purely numerical
   conditions on 
   $r$ and the $m_i$  such that $L-K$ there is ample and
   base-point-free. 

   \begin{corollary}
     Fix multiplicities  $m_1\leq m_2\leq \dots\leq m_n$.
     Let $H$ denote the class of a line in $\pp^2$ and assume that, for
     $p_1,\dots,p_n$ general in $\pp^2$, the linear
     system $M=rH-m_1p_1-\dots-m_np_n$ is defective, i.e.
     \begin{displaymath}
       \dim|M|>
       \max\left\{-1,\binom{r+2}2-\sum_{i=1}^n\binom{m_i+1}2\right\}.
     \end{displaymath}
     Let $\pi:S\longrightarrow\pp^2$ be the blowing up of $\pp^2$ at the
     points $p_2,\dots,p_n$ and set
     $L:= r\pi^*H-m_2E_2-\dots-m_nE_n$, where $E_i=\pi^*(p_i)$ is the
     i-th exceptional divisor. Assume that $L$ is very ample on $S$,
     of the expected dimension  $\binom{r+2}2-\sum_{i=2}^n\binom{m_i+1}2$,
     and
     that $L-K$ is ample and base-point-free on $S$, with $(L-K)^2>16$.
     Assume, finally, $m_1\leq 3$.

     Then $m_1=3$ and the general element of $M$
     is non-reduced. Moreover $L$ embeds $S$ as a ruled surface.
   \end{corollary}
   \begin{proof}
     Just apply the Main Theorem \ref{thm:aim1} to the pair $(S,L)$.
   \end{proof}

   The reader can easily check that the previous result is exactly the
   translation
   of Segre's and Harbourne--Hirschowitz's conjectures on defective linear
   systems in the plane, for the case of a \emph{minimally} defective
   system with lower multiplicity $3$. The $(-1)$--curve predicted by
   the Harbourne--Hirschowitz conjecture, in this situation, is just the
   pull-back of a line of the ruling.
   Thus, although in a partial situation, we get new evidence for the
   conjecture, at least when the {\it minimal} multiplicity imposed at the
   points is $3$.

   \smallskip
   The paper is organised as follows.

   The case where $\beta$ is not generically finite
   is pointed out in Theorem \ref{thm:notfinite} in Section
   \ref{sec:triplecomponents}.
   In Section \ref{sec:equimultiple} we reformulate the problem as an
   $h^1$-vanishing problem.  The Sections
   \ref{sec:construction} to \ref{sec:ruled}  are
   devoted to the proof of the main result: in Section
   \ref{sec:construction} we
   use Serre's construction and Bogomolov instability in order to show
   that triple-point defectiveness leads to the existence of very
   special divisors $A$ and $B$ on our surface; in Section \ref{sec:zero} we
   show that $|B|$ has no fixed component; in Section
   \ref{sec:generalcase} we then
   list properties of $B$ and we use these in Section \ref{sec:ruled} to
   classify the triple-point defective surfaces.

   The authors wish to thank the referee, who pointed out the possibility
   of weakening one assumption in a preliminary version of the main theorem. 

   \section{Triple Components}\label{sec:triplecomponents}

   In this section, we consider what happens when, in Diagram \eqref{VVV},
    the general fibre of $\beta$ is positive-dimensional,
   in other words, when  the general member of $V_3$ contains a triple
   component through $p$. 

   This case has been investigated (and essentially solved) in \cite{Cas22}, 
   and then re\-phrased in modern language in  \cite{FI01} and \cite{BC05}. 
   
   Although not strictly necessary for the sequel, as our arguments 
   do not make any use of the generic finiteness of $\beta$, 
   (and so we will not assume this), for the sake of completeness
   we recall in this section some example and 
   the classification of pairs $(S,L)$
   which are triple-point defective, and such that a general curve $L_p\in|L-3p|$
   has a triple component through $p$.
   \smallskip

   The family $\kl_3$ of pairs $(L,p)\in |L|\times S$ where $L\in |L-3p|$
   has dimension bounded below by $\dim|L|-4$, and in Remark \ref{L3}
   it has been pointed out that $(S,L)$ is triple-point defective exactly when $\alpha$ is
   dominant and the bound is not attained.
   
   Notice however that $\dim|L|-4$ is {\it not} necessarily a bound for the dimension
   of the subvariety $V_3\subset |L|$, the image of $\kl_3$ under $\beta$.
   The following example (exploited in \cite{LM02}) shows that one may have $\dim(V_3)<\dim|L|-4$ 
   even when $(S,L)$ is \emph{not} triple-point defective.
  
   \begin{example}[(see \cite{LM02})]
      Let $S$ be the blowing up
     of $\pp^2$ at $8$ general points $q_1,\dots,q_8$ and $L$ corresponds to the system
     of curves of degree nine in $\pp^2$, with a triple point at each $q_i$.

     $\dim|L|=6$, but for $p\in S$ general, the unique divisor in
     $|L-3p|$ coincides with the cubic plane curve through $q_1,\dots,q_8,p$,
     counted three times. As there exists only a (non-linear) one-dimensional family
     of such divisors in $|L|$, then $\dim(V_3)=1<\dim|L|-4$. On the other hand,
     these divisors have a triple component, so that the general fibre of
     $\beta$ has dimension one, hence $\dim(\kl_3)=2=\dim|L|-4$.
   \end{example}
   
   The classification of triple-point defective pairs $(S,L)$ for
   which the map $\beta$ is not generically 
   finite is the following.

   \begin{theorem} \label{thm:notfinite}
     Suppose that $(S,L)$ is triple-point defective. Then for $p\in S$ general, 
     the general member of $|L-3p|$
     contains a triple  component through $p$ if and only if
     $S$ lies in a three dimensional scroll $W$ containing a one
     dimensional family of planes, and moreover $W$ is 
     developable, i.e. the tangent space to $W$ is constant along the planes.
   \end{theorem}
   \begin{proof} 
     First, since we assume that $S$ is triple-point defective and
     embedded in $\pp^r$ via $L$, then the hyperplanes
     $\pi$ that meet $S$ in a divisor $H=S\cap \pi$ with a triple point at a general
     $p\in S$, intersect in a $\pp^4$. Thus we may project down $S$ to $\pp^5$ and work
     with the corresponding surface.

     In this setting, through a general $p\in S$ one has only one hyperplane $\pi$ with
     a triple contact, and $\pi$ has a triple contact with $S$ along the fibre $C$ of
     $\beta$. Thus $V_3$ is a curve.

     If $H', H''$ are two consecutive infinitesimally near points to $H$ on $V_3$, then
     $C$ also belongs to $H\cap H'\cap H''$. Thus $C$ is a plane curve and $S$ is fibred
     by a $1$-dimensional family of plane curves. This determines the three dimensional scroll $W$.

     The tangent line to $V_3$ determines in $(\pp^5)^*$ a pencil of hyperplanes which are
     tangent to $S$ at any point of $C$, since this is the infinitesimal deformation
     of a family of hyperplanes with a triple contact along any point of $C$. Thus
     there is a $\pp^4=H_C$ which is tangent to $S$ along $C$.

     Assume that $C$ is not a line. Then $C$ spans a $\pp^2=\pi_C$ fibre of $W$, moreover the
     tangent space to $W$ at a general point of $C$ is spanned by $\pi_C$ and $T_{S,P}$, hence
     it is constantly equal to $H_C$. Since $C$ spans $\pi_C$, then it turns out that
     the tangent space to $W$ is constant at any point of $\pi_C$, i.e. $W$ is developable.

     When $C$ is a line, then arguing as above one finds that all the  tangent planes
     to $S$ along $C$ belong to the same $\pp^3$. This is enough to conclude
     that $S$ sits in some developable $3$-dimensional scroll.

     Conversely, if $S$ is contained in the developable scroll $W$, then at a general point
     $p$, with local coordinates $x,y$, the tangent space $t$ to $W$ at $p$ contains the
     derivatives $p, p_x, p_y, p_{xx}, p_{xy}$ (here $x$ is the direction of the tangent
     line to $C$). Thus the $\pp^4$ spanned by $t,p_{yy}$ intersects $S$ in a triple curve
     along $C$.
   \end{proof}

   \section{The Equimultiplicity Ideal}\label{sec:equimultiple}

   If $L_p$ is a curve in $|L-3p|$ we denote by $f_p\in\C\{x_p,y_p\}$ an
   equation of $L_p$ in local coordinates $x_p$ and $y_p$ at $p$.
   If $\mult_p(L_p)=3$, the ideal sheaf $\kj_{Z_p}$ whose stalk 
   at $p$ is the equimultiplicity   ideal
   \begin{displaymath}
     \kj_{Z_p,p}=\left\langle\frac{\partial
         f_p}{\partial x_p},\frac{\partial f_p}{\partial
         y_p}\right\rangle + \langle x_p,y_p\rangle^3
   \end{displaymath}
   of $f_p$ defines a zero-dimensional scheme
   $Z_p=Z_p(L_p)$ concentrated at $p$, and the tangent space
   $T_{(L_p,p)}({\kl_3})$ of $\kl_3$ at $(L_p,p)$ satisfies (see
   \cite{Mar06} Example 10)
   \begin{displaymath}
     T_{(L_p,p)}({\kl_3})\cong \big(H^0\big(S,\kj_{Z_p}(L_p)\big)/H^0(S,\ko_S)\big)\oplus\kk,
   \end{displaymath}
   where $\kk$ is zero unless $L_p$ is unitangential at $p$, in which
   case $\kk$ is a one-dimensional  vector space.

   In particular, $\kl_3$ is smooth at $(L_p,p)$ of the expected
   dimension (see \cite{Mar06} Proposition 11)
   \begin{displaymath}
     \expdim(\kl_3)=\dim|L|-4
   \end{displaymath}
   as soon as
   \begin{displaymath}
     h^1\big({S},\kj_{Z_p}(L)\big)=0.
   \end{displaymath}
   We thus have the following proposition.

   \begin{proposition}\label{prop:h1vanishing}
     Suppose that $\alpha$ is surjective,
     then $(S,L)$ is not triple-point defective if 
     \begin{displaymath}
       h^1\big({S},\kj_{Z_p}(L)\big)=0
     \end{displaymath}
     for general $p\in S$ and $L_p\in|L|$ with $\mult_p(L_p)=3$.

     Moreover, if $L$ is non-special, i.e.\ if $h^1(S,L)=0$, the above $h^1$-vanishing is also
     necessary for the non-triple-point-defectiveness of $(S,L)$.
   \end{proposition}

   Note that by Kodaira vanishing $L$ is non-special whenever $L-K$ is
   ample.

   \section{The Basic Construction}\label{sec:construction}

   \medskip
   \begin{center}
     \framebox[12cm]{
       \begin{minipage}{11.3cm}
         \medskip
         \emph{From now on we assume that for $p\in S$ general
           $\exists\;L_p\in|L|$ s.t.
           $$h^1\big({S},\kj_{Z_p}(L)\big)\not=0.$$}
         \vspace*{-2ex}
       \end{minipage}
       }
   \end{center}
   \medskip

   Then by Serre's construction for a subscheme $Z'_p\subseteq Z_p$ with
   ideal sheaf $\kj_p=\kj_{Z_p'}$ of minimal length such that
   $h^1\big(S,\kj_p(L)\big)\not=0$ there is a rank two bundle $\ke_p$
   on $S$ and a section $s\in H^0(S,\ke_p)$ whose $0$-locus is $Z'_p$,
   giving the exact sequence
   \begin{equation}\label{eq:vectorbundle}
     0\rightarrow
     \ko_S\rightarrow\ke_p\rightarrow\kj_p(L-K)\rightarrow 0.
   \end{equation}
   The Chern classes of $\ke_p$ are
   \begin{displaymath}
     c_1(\ke_p)=L-K\;\;\;\mbox{ and }\;\;\;
     c_2(\ke_p)=\length(Z'_p).
   \end{displaymath}
   Moreover, $Z'_p$ is automatically a complete
   intersection.

   We would now like to understand what $\kj_p$ is depending on
   $\jet_3(f_p)$, which in suitable local coordinates will be one
   of those in Table \eqref{eq:3jets}.
   For this we first of all note that the very ample
   divisor $L$ separates all subschemes of $Z_p$ of length
   at most two. Thus $Z'_p$ has length at least $3$, and due to Lemma
   \ref{lem:ideals} below we are in one of the following situations:

   \begin{equation}\label{eq:3jets}
     \begin{array}{|c|c|c|c|c|}
       \hline
       \jet_3(f_p)&\kj_{Z_p,p}&\length(Z_p)&\kj_p=\kj_{Z'_p,p}&c_2(\ke_p)\\
       \hline\hline
       x_p^3-y_p^3 & \langle x_p^2,y_p^2 \rangle & 4 & \langle x_p^2,y_p^2 \rangle & 4\\\hline
       x_p^2y_p  &\langle x_p^2, x_py_p,y_p^3\rangle & 4 &\langle x_p,y_p^3 \rangle & 3\\\hline
       x_p^3 & \langle x_p^2,x_py_p^2,y_p^3\rangle & 5&\langle x_p^2,y_p^2 \rangle & 4\\\hline
       x_p^3 & \langle x_p^2,x_py_p^2,y_p^3\rangle & 5&\langle x_p,y_p^3 \rangle & 3\\\hline
     \end{array}
   \end{equation}

   \bigskip

   \begin{lemma}\label{lem:ideals}
     If $f\in R=\C\{x,y\}$ with $\jet_3(f)\in\{x^3-y^3,x^2y,x^3\}$, and if $I=\langle
     g,h\rangle\lhd R$ such that
     $\dim_\C(R/I)\geq 3$ and $\big\langle \frac{\partial f}{\partial
       x},\frac{\partial f}{\partial y}\big\rangle+\langle
     x,y\rangle^3\subseteq I$, then we may assume that we are in
     one of the following cases:
     \begin{enumerate}
     \item $I=\langle x^2,y^2\rangle$ and
       $\jet_3(f)\in\{x^3-y^3,x^3\}$, or
     \item $I=\langle x,y^3\rangle$ and $\jet_3(f)\in\{x^2y,x^3\}$.
     \end{enumerate}
   \end{lemma}
   \begin{proof}
     If $>$ is any \emph{local degree} ordering on $R$, then the
     Hilbert-Samuel functions of $R/I$ and of $R/L_>(I)$ coincide,
     where $L_>(I)$ denotes the leading ideal of $I$ (see
     e.g. \cite{GP02} Proposition 5.5.7). In particular,
     $\dim_\C(R/I)=\dim_\C(R/L_>(I))$ and thus
     \begin{displaymath}
       L_>(I)\in\big\{\langle x^2,xy^2,y^3\rangle,\langle x^2,xy,y^2\rangle,\langle
       x^2,xy,y^3\rangle,\langle x^2,y^2\rangle,\langle
       x,y^3\rangle\},
     \end{displaymath}
     since $\langle x^2,xy^2,y^3\rangle\subset I$.

     Taking $>$, for a moment, to be the local
     degree ordering on $R$ with $y>x$ we deduce at once that $I$ does
     not contain any power series with a linear term in $y$. For the
     remaining part of the proof $>$ will be the local degree ordering
     on $R$ with $x>y$.

     \underline{1st Case:} $L_>(I)=\langle x^2,xy^2,y^3\rangle$ or 
    $L_>(I)=\langle x^2,xy,y^2\rangle$.
       Thus the graph of the slope $H^0_{R/I}$ of the Hilbert-Samuel  function of $R/I$
       would be as shown in Figure \ref{fig:fp-histogram}, which
       contradicts the fact that $I$ is a complete intersection due to
       \cite{Iar77} Theorem 4.3.
     \begin{figure}[h]
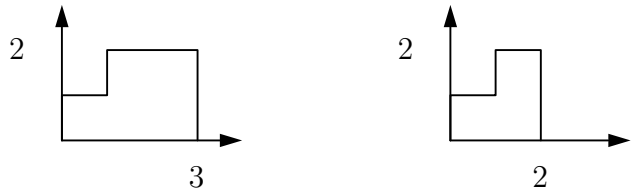

       \smallskip
       \begin{center}
         \begin{texdraw}
           %% Fat Point Histogram of $H^0_{R/I}$. %%
           \drawdim cm
           \setunitscale 0.6
           \arrowheadtype t:F
           \arrowheadsize l:0.5 w:0.3
           \move (0 0)
           \rlvec (0 1) \rlvec (1 0)
           \rlvec (0 1) \rlvec (2 0)
           \rlvec (0 -2)
           \move (0 0) \avec (0 3)
           \move (0 0) \avec (4 0)
           \move (0 0)
           \rmove (3 -0.6) \textref h:C v:T \htext{$3$}
           \move (-0.8 2)  \textref h:R v:C \htext{$2$}
         \end{texdraw}
         \hspace*{2cm}
         \begin{texdraw}
           %% Fat Point Histogram of $H^0_{R/I}$. %%
           \drawdim cm
           \setunitscale 0.6
           \arrowheadtype t:F
           \arrowheadsize l:0.5 w:0.3
           \move (0 0)
           \rlvec (0 1) \rlvec (1 0)
           \rlvec (0 1) \rlvec (1 0)
           \rlvec (0 -2)
           \move (0 0) \avec (0 3)
           \move (0 0) \avec (4 0)
           \move (0 0)
           \rmove (2 -0.6) \textref h:C v:T \htext{$2$}
           \move (-0.8 2)  \textref h:R v:C \htext{$2$}
         \end{texdraw}
         \medskip
         \caption{The graphs of $H^0_{R/\langle x^2,xy^2,y^3\rangle}$
           respectively of $H^0_{R/\langle x^2,xy,y^2\rangle}$.}
         \label{fig:fp-histogram}
       \end{center}
     \end{figure}

     \underline{2nd Case:} $L_>(I)=\langle x^2,xy,y^3\rangle$.
     Then we may assume
     \begin{displaymath}
       g=x^2+\alpha\cdot y^2+h.o.t.\;\;\;\mbox{ and }\;\;\; h=xy+\beta\cdot y^2+h.o.t..
     \end{displaymath}
     Since $x^2\in I$ there are power series $a,b\in R$ such that
     \begin{displaymath}
       x^2=a\cdot g+b\cdot h.
     \end{displaymath}
     Thus the leading monomial of $a$ is one, $a$ is a unit and
     $g\in\langle x^2,h\rangle$. We may therefore assume that
     $g=x^2$. Moreover, since the intersection multiplicity of $g$ and
     $h$ is $\dim_\C(R/I)=4$, $g$ and $h$ cannot have a common tangent
     line in the origin, i.\ e.\ $\beta\not=0$. Thus, since $g=x^2$,
     we may assume that $h=xy+y^2\cdot u$ with $u=\beta+h.o.t$ a unit.

     In new coordinates $\widetilde{x}=x\cdot\sqrt{u}$ and
     $\widetilde{y}=y\cdot\frac{1}{\sqrt{u}}$ we have
     \begin{displaymath}
       I=\langle \widetilde{x}^2,\widetilde{x}\widetilde{y}+\widetilde{y}^2\rangle.
     \end{displaymath}
     Note that by the coordinate change $\jet_3(f)$ only changes by a constant,
     that $\frac{\partial f}{\partial \widetilde{x}},\frac{\partial
       f}{\partial \widetilde{y}}\in I$ and that $\langle
     \widetilde{x},\widetilde{y}\rangle^3\subset I$, but
     $\widetilde{x}\widetilde{y},\widetilde{y}^2\not\in I$. Thus
     $\jet_3(f)=x^3$.

     Setting now $\bar{x}=\widetilde{x}$ and
     $\bar{y}=\widetilde{x}+2\widetilde{y}$, then
     $\bar{y}^2=\widetilde{x}^2+4\cdot
     (\widetilde{x}\widetilde{y}+\widetilde{y}^2)\in I$ and thus,
     considering colengths,
     \begin{displaymath}
       I=\langle \bar{x}^2,\bar{y}^2\rangle.
     \end{displaymath}
     Moreover, the $3$-jet of $f$ does not change with respect
     to the new coordinates, so that we may assume we worked with
     these from the beginning.

     \underline{3rd Case:} $L_>(I)=\langle x^2,y^2\rangle$.
     Then we may assume
     \begin{displaymath}
       g=x^2+\alpha\cdot xy+h.o.t.\;\;\;\mbox{ and }\;\;\; h=y^2+h.o.t.
     \end{displaymath}
     As in the second case we deduce that w.l.o.g.\ $g=x^2$ and
     thus $h=y^2\cdot u$, where $u$ is a unit. But then $I=\langle
     x^2,y^2\rangle$.

     \underline{4th Case:} $L_>(I)=\langle x,y^3\rangle$.
     Then we may assume
     \begin{displaymath}
       g=x+h.o.t.\;\;\;\mbox{ and }\;\;\; h=y^3+h.o.t.
     \end{displaymath}
     since there is no power series in $I$ involving a linear term in
     $y$. In new coordinates $\widetilde{x}=g$ and
     $\widetilde{y}=y$ we have
     \begin{displaymath}
       I=\big\langle \widetilde{x},\widetilde{h}\big\rangle,
     \end{displaymath}
     and we may assume that $\widetilde{h}=\widetilde{y}^3\cdot u$, where $u$ is a
     unit only depending on $\widetilde{y}$. Hence, $I=\langle
     \widetilde{x},\widetilde{y}^3\rangle$. Moreover, the $3$-jet
     of $f$ does not change with respect
     to the new coordinates, so that we may assume we worked with
     these from the beginning.
   \end{proof}

   \medskip
   \begin{center}
     \framebox[11cm]{
       \begin{minipage}{10cm}
         \medskip
         From now on we assume that $(L-K)^2>16$.
         \medskip
       \end{minipage}
       }
   \end{center}
   \bigskip

   Thus
   \begin{displaymath}
     c_1(\ke_p)^2-4\cdot c_2(\ke_p)>0,
   \end{displaymath}
   and hence $\ke_p$ is Bogomolov unstable. The Bogomolov instability
   implies the existence of a unique divisor $A_p$ which destabilises
   $\ke_p$. (See e.\ g.\ \cite{Fri98} Section 9, Corollary 2.) In
   other words, setting $B_p=L-K-A_p$, i.\ e.\
   \begin{equation}
     \label{eq:AB:0}
     A_p+B_p=L-K,
   \end{equation}
   there is an immersion
   \begin{equation}
     0\rightarrow\ko_S(A_p)\rightarrow\ke_p
   \end{equation}
   where $(A_p-B_p)^2\geq c_1(\ke_p)^2-4\cdot c_2(\ke_p)>0$ and 
   $(A_p-B_p){\intmult}H>0$ for every ample $H$. The same construction was
   considered in \cite{BFS89} and with their Proposition 1.4 it follows:
   \begin{enumerate}
   \item $\ke_p(-A_p)$ has a global section that vanishes along a
     subscheme $\widetilde{Z}_p$ of codimension $2$ and which gives
     rise to a short exact sequence:
     \begin{equation}\label{eq:AB:1}
       0\rightarrow\ko_S(A_p)\rightarrow\ke_p\rightarrow\kj_{\widetilde{Z}_p}(B_p)\rightarrow 0.
     \end{equation}
   \item The divisor $B_p$ is effective and we may assume that
     $Z'_p\subset B_p$.
   \item The divisors $A_p$ and $B_p$ satisfy the following numerical condition:
     \begin{equation}
       \label{eq:AB:2}
       \length(Z'_p)\geq A_p{\intmult}B_p \geq B_p^2+1.
     \end{equation}
   \item $A_p-B_p$ and $A_p$ are  big.
   \end{enumerate}

   \tom{
   \begin{lemma}\label{lem:big}
     Let $R$ be a big divisor. If $M$ is big and nef or if $M$ is an irreducible curve with
     $M^2\geq 0$ or if $M$ is an effective divisor without fixed component, then $R{\intmult}M>0$.
   \end{lemma}
   \begin{proof}
     If $R$ is big, then $\dim|k\cdot R|$ grows with $k^2$.
     Thus for
     $k>>0$ we can write $k\cdot R=N'+N''$ where $N'$ is ample
     and $N''$ effective (possibly zero). To see this, note that
     for $k>>0$ we can write $|k\cdot R|=|N'|+N''$, where $N''$ is the
     fixed part of $|k\cdot R|$ and $N'\cap C\not=\emptyset$ for every
     irreducible curve $C$. Then apply the Nakai-Moishezon Criterion
     to $N'$ (see also \cite{Tan04}). 

     Analogously, if $M$ is big and nef, for
     $j>>0$ we can write $j\cdot M=M'+M''$ where $M'$ is ample and
     $M''$ is effective. Therefore,
     \begin{displaymath}
       R{\intmult}M=\frac{1}{kj}\cdot \big(N'{\intmult}M'+N'{\intmult}M''+N''{\intmult}M)>0,
     \end{displaymath}
     since $N'{\intmult}M'>0$, $N'{\intmult}M''\geq 0$ and $N''{\intmult}M\geq 0$.

     Similarly, if $M$ is irreducible and has non-negative
     self-inter\-section, then
     \begin{displaymath}
       R{\intmult}M=\frac{1}{k}\cdot(N'{\intmult}M+N''{\intmult}M)>0.
     \end{displaymath}

     And if $M$ is effective without fixed component, we can apply the
     previous argument to every component of $M$.
   \end{proof}
   }

   Now let $p$ move freely in $S$. Accordingly the scheme
   $Z'_p$ moves, hence the effective divisor $B_p$ containing $Z'_p$
   moves in an algebraic family $\kb\subseteq |B|_a$ which is the
   closure of $\{B_p\;|\;p\in S, L_p\in|L-3p|, \mbox{ both general}\}$
   and which covers $S$. 
   A priori this family $\kb$ might
   have a \emph{fixed part} $C$, \label{page:fixedpart} so that for
   general $p\in S$ there is an 
   effective divisor $D_p$ moving in a fixed-part free algebraic
   family $\kd\subseteq |D|_a$  such that
   \begin{displaymath}
     B_p=C+D_p.
   \end{displaymath}
   Whenever we only refer to the algebraic class of $A_p$ respectively
   $B_p$ respectively $D_p$ we will write $A$ respectively $B$
   respectively $D$ for short.

   \medskip
   \begin{center}
     \framebox[11cm]{
       \begin{minipage}{10cm}
         \medskip
         For these considerations we assume, of course, that $\length(Z'_p)$ is constant for
         $p\in S$ general, so either $\length(Z'_p)=3$ or $\length(Z'_p)=4$.
         \medskip
       \end{minipage}
       }
   \end{center}
   \bigskip

   \section{$C=0$.}\label{sec:zero}

   Our first aim is to show that actually $C=0$ (see Lemma
   \ref{lem:C}). But in order to do so
   we first have to consider the
   boundary case that $A_p{\intmult}B_p=\length(Z_p')$.

   \begin{proposition}\label{prop:splitting}
     If $A_p{\intmult}B_p=\length(Z'_p)$, then
     there exists a non-trivial global section $0\not=s\in
     H^0\big(B_p,\kj_{Z_p'/B_p}(A_p)\big)$ whose zero-locus is
     $Z_p'$.

     In particular, $A_p{\intmult}D_p=A_p{\intmult}B_p=\length(Z_p')$ and $A_p{\intmult}C=0$.
   \end{proposition}
   \begin{proof}
     By Sequence \eqref{eq:AB:1} we have
     \begin{displaymath}
       A_p{\intmult}B_p=\length(Z'_p)=c_2(\ke_p)=A_p{\intmult}B_p+\length\big(\widetilde{Z}_p\big).
     \end{displaymath}
     Thus $\widetilde{Z}_p=\emptyset$.

     If we  merge the sequences \eqref{eq:vectorbundle},
     \eqref{eq:AB:1}, and the structure sequence of $B$ twisted by $B$
     we obtain the  exact commutative diagram in Figure~\ref{fig:AB},
     \begin{figure}[h]
       \begin{displaymath}
         \xymatrix{
           && 0 \ar[d] & 0 \ar[d] &\\
           &0\ar[r]\ar[d]& \ko_S \ar[r]\ar[d] & \ko_S \ar[r]\ar[d] &0 \\
           0\ar[r]& \ko_S(A_p) \ar[r]\ar[d] & \ke_p \ar[r]\ar[d]&\ko_S(B_p)\ar[r]\ar[d] &0 \\
           0\ar[r]& \ko_S(A_p) \ar[r]\ar[d] & \kj_{Z_p'/S}(A_p+B_p) \ar[r]\ar[d]&\ko_{B_p}(B_p)\ar[r]\ar[d] &0 \\
           &0&0&0
         }
       \end{displaymath}\centering

       \caption{The diagram showing $\ko_{B_p}=\kj_{Z_p'/B_p}(A_p)$.}
       \label{fig:AB}
     \end{figure}
     where $\ko_{B_p}(B_p)=\kj_{Z_p'/B_p}(A_p+B_p)$, or equivalently
     $\ko_{B_p}=\kj_{Z_p'/B_p}(A_p)$.
     Thus from the rightmost column we get a non-trivial global
     section, say $s$, of this bundle 
     which vanishes precisely at $Z_p'$,
     since $Z_p'$ is the zero-locus of the monomorphism of vector
     bundles $\ko_S\hookrightarrow \ke_p$. However, since $p$ is
     general we have that $p\not\in C$ and thus the restriction
     $0\not=s_{|D_p}\in H^0\big(D_p,\kj_{Z_p'/D_p}(A_p)\big)$ and it still vanishes
     precisely at $Z_p'$. Thus $A_p{\intmult}D_p=\length(Z_p')=A_p{\intmult}B_p$, and $A_p{\intmult}C=A_p{\intmult}B_p-A_p{\intmult}D_p=0$.
   \end{proof}

   We next want to show that positive self-intersection of $B$ imposes
   hard restrictions.

   \begin{lemma}\label{lem:B^2positive}
     $B^2\leq 2$ and if $B^2\in \{1,2\}$ then $A{\intmult}B=\length(Z'_p)=4$.
   \end{lemma}
   \begin{proof}
     We may suppose that $B^2>0$. 
     By \eqref{eq:AB:2} we know that $4\geq
     A{\intmult}B>B^2$ and by assumption $(A+B)^2\geq 17$, so that
     \begin{displaymath}
       A^2=(A+B)^2-2\cdot A{\intmult}B-B^2\geq 17-8-3>0
     \end{displaymath}
     and the Hodge Index
     Theorem \tom{ (see e.g.\ \cite{BHPV04}) } gives
     \begin{displaymath}
       (A{\intmult}B)^2\geq A^2\cdot B^2
       \geq (17-2\cdot A{\intmult}B-B^2)\cdot B^2.
     \end{displaymath}
     But then $B^2\geq 3$ leads to the contradiction $16\geq
     18$. Similarly, $A{\intmult}B\leq 3$ leads to $9\geq (11-B^2)\cdot B^2$
     which is neither for $B^2=1$ nor for $B^2=2$ fulfilled. This
     shows that $A{\intmult}B=4$, and thus by \eqref{eq:AB:2} also $\length(Z'_p)=4$.
   \end{proof}

   Even though we do not know whether $\kb$ has a fixed part or not, we
   can get some information about the moving part $\kd$. 

   \begin{lemma}\label{lem:DD}
     Let $p\in S$ be general and suppose $\length(Z'_p)=4$.
     \begin{enumerate}
     \item If $D_p$ is irreducible, then $\dim(\kd)\geq 2$ and
       $D_p^2\geq 3$.
     \item If $D_p$ is reducible but the part containing $p$ is reduced,
       then
       either $D_p$ has a component singular in $p$ and $D_p^2\geq 5$
       or at least two components of $D_p$ pass through $p$ and
       $D_p^2\geq 2$.
     \item If $D_p^2\leq 1$, then $D_p=k\cdot E_p$ where $k\geq 2$,
       $E_p$ is irreducible and $E_p^2=0$. In particular, $D_p^2=0$.
     \end{enumerate}
   \end{lemma}
   \begin{proof}
     \begin{enumerate}
     \item
       If $D_p$ is irreducible, then $\dim(\kd)\geq 2$, since $D_p$,
       containing $Z'_p$, is
       singular in $p$ by Table \eqref{eq:3jets} and since $p\in S$ is
       general. If through $p\in S$
       general and a general $q\in D_p$ there is another $D'\in\kd$,
       then due to the irreducibility of $D_p$
       \begin{displaymath}
         D_p^2=D_p{\intmult}D'\geq \mult_p(D_p)+\mult_q(D_p)\geq 3.
       \end{displaymath}
       Otherwise, $\kd$ is a two-dimensional involution whose
       general element is irreducible, so that by \cite{CC02} Theorem
       5.10 $\kd$ must be a linear system. This, however,
       contradicts the Theorem of Bertini, since the general element
       of $\kd$ would be singular.
     \item
       Suppose $D_p=\sum_{i=1}^k E_{i,p}$ is reducible but the part
       containing $p$ is reduced. Since
       $D_p$ has no fixed component and $p$ is general, each $E_{i,p}$
       moves in an at least one-dimensional family. In particular
       $E_{i,p}^2\geq 0$.

       If some $E_{i,p}$, say $i=1$, would be singular in $p$ for
       $p\in S$ general we could argue as above that $E_{1,p}^2\geq
       3$. Moreover, either $E_{2,p}$ is algebraically equivalent to
       $E_{1,p}$ and $E_{2,p}^2\geq3$, or $E_{1,p}$ and $E_{2,p}$
       intersect properly, since both vary in different, at least
       one-dimensional families. In any case we have
       \begin{displaymath}
         D_p^2\geq (E_{1,p}+E_{2,p})^2\geq 5.
       \end{displaymath}
       Otherwise, at least two components, say
       $E_{1,p}$ and $E_{2,p}$ pass through $p$, since $D_p$ is singular
       in $p$ and no component passes through $p$ with higher
       multiplicity. Hence, $E_{1,p}{\intmult}E_{2,p}\geq 1$ and therefore
       \begin{displaymath}
         D_p^2\geq 2\cdot E_{1,p}{\intmult}E_{2,p}\geq 2.
       \end{displaymath}
     \item
       From the above we see that $D_p$ is not reduced in $p$. Let therefore
       $D_p\equiv_a kE_p+E'$
       where $k\geq 2$, $E_p$ passes through $p$ and $E'$ does not contain any
       component algebraically
       equivalent to $E_p$\tom{ (note, if $E'$ contains a component
         algebraically equivalent to $E_p$ we may replace it by $E_p$
         without changing anything)}.

       Suppose $E'\not=0$.\label{eq:DD:0} Since $D_p$ has
       no fixed component both, $E_p$ and $E'$ vary in an at least one
       dimensional family covering $S$ and must therefore intersect
       properly. In particular, $E_p{\intmult}E'\geq 1$ and $1\geq D_p^2\geq 2k\cdot
       E_p{\intmult}E'\geq 4$. Thus, $E'=0$.

       We therefore may assume that
       $D_p=kE_p$ with $k\geq 2$. Then $0\leq E_p^2=\frac{1}{k^2}\cdot
       D_p^2\leq \frac{1}{4}$, which leaves only the possibility $E_p^2=0$,
       implying also $D_p^2=0$.
     \end{enumerate}
   \end{proof}

   The following observations on the self intersection number of
   irreducible curves embedded via $L-K$ in our situation is an
   important tool in the proof that the fixed part $C$ does not
   exist. 

   \begin{lemma}\label{lem:curves}
     Suppose that $R\subset S$ is an irreducible curve, $L$ is very
     ample, and $L-K$ is base-point-free on $S$.
     \begin{enumerate}
     \item If $(L-K){\intmult}R=1$, then $R$ is smooth, rational and
       $R^2\leq -2$.
     \item If $(L-K){\intmult}R=2$, then one of the following two cases occurs:
       \begin{enumerate}
       \item $R$ is smooth and rational with $R^2\leq -1$, or
       \item $|L-K|$ induces a $\mathfrak{g}^1_2$ on $R$ and $L+R$
         does not separate the points of this $\mathfrak{g}^1_2$.
       \end{enumerate}
       In any case, if $R$ moves in a one dimensional algebraic family, then $R^2\not=0$.
     \end{enumerate}
   \end{lemma}
   \begin{proof}
     Since $|L-K|$ is base-point-free it defines a morphism
     \begin{displaymath}
       \varphi_{|L-K|}:S\longrightarrow \pp^n
     \end{displaymath}
     and if $C=\varphi_{|L-K|}(R)$ and $\varphi:R\longrightarrow C$ denotes
     the restriction of $\varphi_{|L-K|}$ then
     \begin{displaymath}
       \deg(\varphi)\cdot\deg(C)=(L-K){\intmult}R.
     \end{displaymath}
     Moreover, by the adjunction formula we know that
     \begin{displaymath}
       p_a(R)=\frac{R^2+R{\intmult}K}{2}+1,
     \end{displaymath}
     and since $L$ is very ample we thus get
     \begin{equation}
       \label{eq:curves:1}
       1\leq L{\intmult}R =(L-K){\intmult}R+R{\intmult}K=(L-K){\intmult}R+2\cdot\big(p_a(R)-1\big)-R^2.
     \end{equation}
     \begin{enumerate}
     \item If $(L-K){\intmult}R=1$, then $C$ is a line in $\pp^n$ and $\varphi$
       is a birational morphism from $R$ to $C$. It thus is an
       isomorphism, and $R$ must be a smooth,
       rational curve. We deduce from \eqref{eq:curves:1}
       \begin{displaymath}
         R^2\leq (L-K){\intmult}R-3=-2.
       \end{displaymath}
     \item If $(L-K){\intmult}R=2$, then either the degree of $\varphi$
       is one or two.

       Suppose first that $\deg(\varphi)=1$.
       Then as above $\varphi$ is a birational morphism and hence an
       isomorphism. $C$ being an irreducible conic it is smooth and
       rational, and so is $R$. We deduce from \eqref{eq:curves:1}
       \begin{displaymath}
         R^2\leq (L-K){\intmult}R-3=-1.
       \end{displaymath}

       Consider now the case $\deg(\varphi)=2$.
       $|L-K|$ cuts out a $\mathfrak{g}^1_2$ on $R$ which induces
         the morphism $\varphi$. Even if $R$ is singular the dualizing
         sheaf on $R$ is given by the restriction of $K+R$, and it
         satisfies the Riemann-Roch formula (see e.g.\ \cite[Ex.~IV.I.9]{Har77}\tom{ or Remark
           \ref{rem:rr}}), i.e. if $\mathfrak{d}$ is any divisor on
         $C$ we have
         \begin{equation}\label{eq:rr}
           h^0(\mathfrak{d})-h^0\big((K+R)_{|R}-\mathfrak{d})=\deg(\mathfrak{d})+1-p_a(R).
         \end{equation}
         Suppose now that $P+Q\in \mathfrak{g}^1_2$ with $P$ and $Q$
         in the smooth part of $C$. Then
         \begin{displaymath}
           h^0\big((K+R)_{|R}-(L+R)_{|R}+P\big)=h^0\big(P-(L-K)_{|R}\big)=h^0(-Q)=0
         \end{displaymath}
         and
         \begin{displaymath}
           h^0\big((K+R)_{|R}-(L+R)_{|R}+P+Q\big)=h^0\big(P+Q-(L-K)_{|R}\big)=h^0(\ko_R)=1.
         \end{displaymath}
         The Theorem of Riemann-Roch \eqref{eq:rr} thus gives
         \begin{displaymath}
           h^0\big((L+R)_{|R}-P\big)=(L+R).R-1+1-p_a(R)
         \end{displaymath}
         and
         \begin{displaymath}
           h^0\big((L+R)_{|R}-P-Q\big)-1=(L+R).R-2+1-p_a(R).
         \end{displaymath}
         Hence
         \begin{displaymath}
           h^0\big((L+R)_{|R}-P\big)=h^0\big((L+R)_{|R}-P-Q\big),
         \end{displaymath}
         i.e.\ each divisor in the linear series induced by $L+R$ on
         $R$ which contains $P$ contains automatically also $Q$. The
         divisors in $|L+R|$ thus do not seperate the points $P$ and
         $Q$. 

         Suppose now that $\dim|R|_a\geq 1$ and $R^2=0$. Then $|R|_a$
         is pencil and induces a fibration of $S$ whose fibres are the
         elements of $|R|_a$ (see \cite{Kei01} App.\ B.1). But then
         $\ko_R(R)$ is trivial (see e.g.\ \cite[Lem.~8.1]{BHPV04}) and
         thus $\ko_R(L+R)=\ko_R(L)$ is very ample, which contradicts
         the fact that it does not separate the points of the
         $\mathfrak{g}^1_2$. 
     \end{enumerate}
   \end{proof}

   \begin{lemma}\label{lem:C}
     The family $\mathcal{B}$ introduced on page \pageref{page:fixedpart}
     has no fixed part. I.e.\ under the assumptions of Section
     \ref{sec:construction} and with the notation there, we have $C=0$.
   \end{lemma}
   \begin{proof}
     Suppose $C\not=0$ and $r$ is the number of irreducible components
     of $C$. Since $\kd$ has no fixed
     component and $A-B$ is big we know \tom{by Lemma \ref{lem:big}} that $(A-B){\intmult}D>0$, so
     that
     \begin{equation}\label{eq:c:0}
       A{\intmult}D\geq B{\intmult}D+1=D{\intmult}C+D^2+1
     \end{equation}
     or equivalently
     \begin{equation}\label{eq:c:1}
       D{\intmult}C\leq A{\intmult}D-D^2-1.
     \end{equation}
     Moreover, since $A+B$ is ample we have
     $r\leq(A+B){\intmult}C=A{\intmult}C+D{\intmult}C+C^2$ and thus
     \begin{equation}\label{eq:c:2}
       A{\intmult}C+D{\intmult}C=(A+B){\intmult}C-C^2\geq r-C^2.
     \end{equation}

     \textbf{1st Case:} $C^2\leq 0$.
     Then \eqref{eq:c:2}
     together with \eqref{eq:c:0} gives
     \begin{equation}\label{eq:c:3}
       A{\intmult}B=A{\intmult}C+A{\intmult}D\geq A{\intmult}C+D{\intmult}C+D^2+1\geq r+(-C^2)+D^2+1\geq 2,
     \end{equation}
     or the slightly stronger inequality
     \begin{equation}
       \label{eq:c:4}
       A{\intmult}B\geq (A+B){\intmult}C +(-C^2)+D^2+1.
     \end{equation}

     \textbf{2nd Case:} $C^2>0$. Then necessarily $B^2>0$ and by Lemma
     \ref{lem:B^2positive} we have $A{\intmult}B=\length(Z'_p)=4$ and 
     \begin{equation}\label{eq:c:5}
       2\geq B^2=D^2+2\cdot C{\intmult}D+C^2\geq 1.
     \end{equation}

     Since all the summands involved in the right hand side of \eqref{eq:c:3} and
     all summands in \eqref{eq:c:5} are non-negative, and since by Lemma \ref{lem:DD}
     the case $D^2=1$ cannot occur when $\length(Z_p')=4$, and since
     by Lemma \ref{lem:B^2positive} $B^2>0$ is impossible when
     $\length(Z'_p)=3$,  we are left considering the
     cases shown in Figure~\ref{fig:c}, where for the additional information (the last
     four columns) we take
     Proposition \ref{prop:splitting}, Lemma \ref{lem:B^2positive} and Lemma \ref{lem:DD} into account.
     \begin{figure}[h]
       \begin{displaymath}
         \begin{array}{|c|c|c|c|c|c||c|c|c|c|}
           \hline
           & \length(Z'_p) & D^2 & C^2 & C{\intmult}D & r & A{\intmult}B & A{\intmult}D & A{\intmult}C & D
           \\\hline\hline
           1)&       4        & 0  & -2  &     & 1 &  4  &  4  &  0  & kE,k\geq 2 \\\hline
           2)&       4        & 0  & -1  &     & 2 &  4  &  4  &  0  & kE,k\geq 2 \\\hline
           3)&       4        & 0  &  0  &     & 3 &  4  &  4  &  0  & kE,k\geq 2 \\\hline
           4)&       4        & 0  & -1  &     & 1 & 3,4 &     &     & kE,k\geq 2 \\\hline
           5)&       4        & 2  &  0  &     & 1 &  4  &  4  &  0  &    \\\hline
           6)&       4        & 0  &  0  &     & 2 & 3,4 &     &     & kE,k\geq 2 \\\hline
           7)&       4        & 0  &  0  &     & 1 &2,3,4&     &     & kE,k\geq 2 \\\hline
           8)&       3        & 0  & -1  &     & 1 &  3  &  3  &  0  &    \\\hline
           9)&       3        & 0  &  0  &     & 2 &  3  &  3  &  0  &    \\\hline
           10)&      3        & 0  &  0  &     & 1 & 2,3 &     &     &    \\\hline\hline
           11)&      4        & 0  &  1  &  0  &   & 4&  4   &  0   & kE,k\geq 2   \\\hline
           12)&      4        & 0  &  2  &  0  &   & 4 & 4   &  0   & kE,k\geq 2   \\\hline
         \end{array}
       \end{displaymath}\centering

       \caption{The cases to be considered.}
       \label{fig:c}
     \end{figure}

     Let us first and for a while consider the situation $\length(Z_p')=4$ and
     $D^2=0$, so that by Lemma \ref{lem:DD}\;\; $D=kE$ for some
     irreducible curve $E$ with $k\geq 2$ and
     $E^2=0$. Applying Lemma \ref{lem:curves} to $E$ we see that
     $(A+B){\intmult}E\geq 3$, and thus
     \begin{equation}\label{eq:c:6a}
       6\leq 3k\leq (A+B){\intmult}D=A{\intmult}D+C{\intmult}D.
     \end{equation}

     If in addition $A{\intmult}D\leq 4$, then \eqref{eq:c:1} leads to
     \begin{equation}
       6\leq 3k\leq A{\intmult}D+C{\intmult}D\leq 4+C{\intmult}D\leq 7,
     \end{equation}
     which is only possible for $k=2$, $C{\intmult}E=1$ and
     \begin{equation}
       C{\intmult}D=k\cdot C{\intmult}E=2.\label{eq:c:6}
     \end{equation}

     In Cases 1, 2  and 3 we have $A{\intmult}D=4$, and we can apply \eqref{eq:c:6}, which
     by \eqref{eq:c:2} then gives the contradiction
     \begin{displaymath}
       2=A{\intmult}C+C{\intmult}D\geq r-C^2=3.
     \end{displaymath}

     If, still under the assumption $\length(Z_p')=4$ and
     $D^2=0$, we moreover assume  $2\geq C^2\geq 0$ then by Lemma \ref{lem:B^2positive}
     \begin{displaymath}
       2\geq B^2=2\cdot C{\intmult}D+C^2\geq 2\cdot C{\intmult}D\geq 0,
     \end{displaymath}
     and thus $C{\intmult}D\leq 1$ and $C{\intmult}D+C^2\leq 2$, which due to \eqref{eq:c:6a}
     implies $A{\intmult}D\geq 5$. But then by Proposition \ref{prop:splitting}
     we have $A{\intmult}B\leq 3$ and hence $A{\intmult}C=A{\intmult}B-A{\intmult}D\leq -2$, which leads
     to the contradiction
     \begin{equation}\label{eq:c:7}
       (A+B){\intmult}C=A{\intmult}C+D{\intmult}C+C^2\leq 0,
     \end{equation}
     since $A+B$ is ample. This rules out the Cases 6,
     7, 11 and 12.

     In Case 4 Lemma \ref{lem:curves} applied to $C$ shows
     \begin{equation}\label{eq:c:8}
       2\leq (A+B){\intmult}C=A{\intmult}C+D{\intmult}C+C^2.
     \end{equation}
     Lemma \ref{lem:B^2positive} implies
     \begin{displaymath}
       2\geq B^2=2\cdot C{\intmult}D+C^2= 2k\cdot C{\intmult}E-1\geq 4\cdot C{\intmult}E-1 \geq -1,
     \end{displaymath}
     which is only possible for $C{\intmult}E=C{\intmult}D=0$. But then \eqref{eq:c:8}
     implies $A{\intmult}C\geq 3$, and since $A$ is big and $E$ is irreducible
     with non-negative self intersection \tom{ by Lemma \ref{lem:big} }
     we get the contradiction
     \begin{displaymath}
       2\leq k\cdot A{\intmult}E\leq A{\intmult}D=A{\intmult}B-A{\intmult}C\leq 1.
     \end{displaymath}

     This finishes the cases where $\length(Z_p')=4$ and $D^2=0$.

     In Cases 5 and 10 we apply Lemma \ref{lem:curves} to the
     irreducible curve $C$ with $C^2=0$ and find
     \begin{displaymath}
       (A+B){\intmult}C\geq 2.
     \end{displaymath}
     In Case 5 Equation \eqref{eq:c:4} then gives the
     contradiction
     \begin{displaymath}
       4= A{\intmult}B \geq 2-C^2+D^2+1=5.
     \end{displaymath}
     In Case 10 we get
     \begin{displaymath}
       3\geq A{\intmult}B \geq (A+B).C-C^2+D^2+1=(A+B).C+1,
     \end{displaymath}
     which shows that
     \begin{equation}\label{eq:c:corr}
       2=(A+B).C=A.C+D.C+C^2
     \end{equation}
     and that $A.B=3=\length(Z_p')$. Then by Proposition
     \ref{prop:splitting} we get $A.C=0$ and by \eqref{eq:c:corr}
     \begin{displaymath}
       D.C=2-A.C-C^2=2,
     \end{displaymath}
     which due to Lemma \ref{lem:B^2positive} leads to the contradiction
     \begin{displaymath}
       2\geq B^2=D^2+2\cdot D.C+C^2=4.
     \end{displaymath}
     In very much the same way we get in Case 8 by Lemma
     \ref{lem:curves} 
     \begin{displaymath}
       (A+B){\intmult}C\geq 2
     \end{displaymath}
     and the contradiction
     \begin{displaymath}
       3= A{\intmult}B\geq 2-C^2+D^2+1=4.
     \end{displaymath}

     It remains to consider Case 9. Here we deduce from \eqref{eq:c:4}
     that
     \begin{displaymath}
       2\geq (A+B){\intmult}C\geq r=2,
     \end{displaymath}
     and hence
     \begin{displaymath}
       2=(A+B){\intmult}C=A{\intmult}C+D{\intmult}C+C^2=D{\intmult}C.
     \end{displaymath}
     But then Lemma \ref{lem:B^2positive} leads to the final contradiction
     \begin{displaymath}
       2\geq B^2=D^2+2\cdot D{\intmult}C+C^2=4.
     \end{displaymath}
   \end{proof}

   It follows that $B_p=D_p$, $\kb=\kd$, and that $B_p$ is nef.

   \section{The General Case}\label{sec:generalcase}

   Let us review the situation and recall some notation. We are
   considering a divisor $L$ such that $L$ is very ample and $L-K$ is
   ample and base-point-free with
   $(L-K)^2>16$, and such that for a general point $p\in S$ the
   general element $L_p\in |L-3p|$ has no triple component through
   $p$ and that the equimultiplicity ideal of $L_p$ in $p$ in suitable
   local coordinates is one of the ideals in Table \eqref{eq:3jets} -- and
   for all $p$ the ideals have the same length. Moreover, we
   know that there is an algebraic family 
   $\kb=\overline{\{B_p\;|\;p\in S\}}\subset |B|_a$ without fixed
   component such that for a general point $p\in S$ 
   \begin{displaymath}
     B_p\in |\kj_{Z'_p/S}(L-K-A_p)|,
   \end{displaymath}
   where $Z'_p$ is contained in the equimultiplicity scheme $Z_p$  of $L_p$ and
   $A_p$ is the unique divisor linearly equivalent to $L-K-B_p$ such that $B_p$ and $A_p$
   destabilise the vector bundle $\ke_p$ in \eqref{eq:vectorbundle}.
   Keeping these notations in mind we can now consider the two cases that
   either $\length(Z'_p)=4$ or $\length(Z'_p)=3$.

   \begin{proposition}\label{prop:length4}
     With the above notation and assumptions, it is impossible that for
     a general point $p\in S$ the length of $Z'_p$ is $\length(Z'_p)=4$.
   \end{proposition}
   \begin{proof}
     In Section \ref{sec:zero} we have shown that $B=D$ is nef, and thus
     Lemma \ref{lem:B^2positive} shows
     \begin{equation}\label{eq:length4:1}
       0\leq B^2\leq 2.
     \end{equation}
     Then, however, Lemma \ref{lem:DD} implies that $B_p$ must be
     reducible. 

     Let us first consider the case that the part of $B_p$ through $p$
     is reduced.  Then
     by Lemma \ref{lem:DD}  and Equation \eqref{eq:length4:1} we know that
     $B_p=E_p+F_p+R$, where $E_p$ and 
     $F_p$ are irreducible and smooth in $p$. In particular,
     $E_p{\intmult}F_p\geq 1$, and thus
     \begin{multline*}
       2= B^2=E_p^2+2\cdot E_p{\intmult}F_p+F_p^2+2\cdot(E_p+F_p){\intmult}R+R^2\\
       \geq 2+2\cdot(E_p+F_p){\intmult}R.
     \end{multline*}
     Since $E_p{\intmult}F_p\geq 1$ and since the components 
     $E_p$ and $F_p$  vary in at
     least one-dimensional families and $R$ has no fixed component, $(E_p+F_p){\intmult}R\geq 1$, unless
     $R=0$. This would however give a contradiction, so
     $R=0$. Therefore necessarily, $B_p=E_p+F_p$, $E_p{\intmult}F_p=1$, and $E_p^2=F_p^2=0$.
     Then by Lemma \ref{lem:curves} $(A+B){\intmult}E_p\geq 3$ and
     $(A+B){\intmult}F_p\geq 3$, so that
     \begin{displaymath}
       4\geq A{\intmult}B\geq (A+B){\intmult}E_p+(A+B){\intmult}F_p-B^2\geq 4
     \end{displaymath}
     implies $E_p{\intmult}A_p=2=F_p{\intmult}A_p$ and
     $(A+B){\intmult}E_p=3=(A+B){\intmult}F_p$. 

     Since $E_p^2=0$ the family $|E|_a$ is a pencil and induces a
     fibration on $S$ (see \cite{Kei01} App.\ B.1). In
     particular, the generic element $E_p$ in $|E|_a$ must be
     smooth (see e.g.\ \cite{BHPV04} p.\ 110). 

     We claim that in $p$ the curve $L_p$ can share at most with one
     of $E_p$ or $F_p$ a common tangent, and it can do so at most with
     multiplicity one. For this consider local coordinates $(x_p,y_p)$ as in the Table 
     \eqref{eq:3jets}. Since $\length(Z_p')=4$ we know that
     $\kj_{Z_p',p}=\langle x_p^2,y_p^2\rangle$ does not contain $x_py_p$, and since
     $B_p=E_p+F_p\in|\kj_{Z_p'}(L-K-A)|$, where $E_p$ and $F_p$ are
     smooth in $p$, we deduce that in local coordinates their
     equations are
     \begin{displaymath}
       x_p+a\cdot y_p+h.o.t.\;\;\;\mbox{ respectively }\;\;\;x_p-a\cdot y_p+h.o.t.,
     \end{displaymath}
     where $a\not=0$. By Table \eqref{eq:3jets} the local equation
     $f_p$ of $L_p$ has either $\jet_3(f_p)=x_p^3$ and has thus no
     common tangent with either $E_p$ or $F_p$, or
     $\jet_3(f_p)=x_p^3-y_p^3$ and it is divisible at most once by one
     of $x_p-ay_p$ or $x_p+ay_p$ .

     In particular, $E_p$ can at most once be a component of $L_p$,
     and we deduce
     \begin{displaymath}
       E_p{\intmult}K_S=E_p{\intmult}L_p-E_p{\intmult}A_p-E_p{\intmult}B_p=E_p{\intmult}L_p-3\geq
       \left\{
         \begin{array}{ll}
           0,&\mbox{ if } E_p\not\subset L_p,\\
           -1, &\mbox{ if } E_p\subset L_p.
         \end{array}
       \right.
     \end{displaymath}
     But then, since the genus is an integer,
     \begin{equation}\label{eq:length4:genus}
       p_a(E_p)=\frac{E_p^2+E_p{\intmult}K_S}{2}+1=\frac{E_p{\intmult}K_S}{2}+1\geq 1.
     \end{equation}

     Fix a general point $p$ in $S$ and a general point $q$ on $E_p$, then
     $E_p=E_q$ since $|E|_a$ is a pencil. Hence,
     \begin{displaymath}
       A_p+F_p\sim_l L-K-E_p=L-K-E_q \sim_l A_q+F_q.
     \end{displaymath}
     Since $A_q.B_q=4=\length(Z'_q)$ by Proposition
     \ref{prop:splitting} there is a global section $s_q\in
     H^0\big(B_q,\kj_{Z'_q/B_q}(A_q)\big)$ whose zero locus is
     $Z'_q$. Restricting $s_q$ to $E_q$ we get a global section of
     $\ko_{E_p}(A_q)=\ko_{E_q}(A_q)$ which cuts out $2q$ on
     $E_p=E_q$. Moreover,  $\ko_{E_p}(F_q)=\ko_{E_q}(F_q)$
     has a global section which cuts out $q$. 
     Thus $\ko_{E_p}(A_p+F_p)=\ko_{E_p}(A_q+F_q)$ has for infinitely
     many points $q$ on $E_p$ a global section
     which cuts out $3q$. The linear
     system $|\ko_{E_p}(A_p+F_p)|$ thus has degree three and contains the
     divisor $3q$ for infinitely many points $q$, and it hence has no
     base point. So it defines a morphism to $\pp^k$, where $k$ is the
     dimension of the linear system. 
     $k$ cannot be one, since otherwise the morphism would have
     infinitely many ramification points. 
     If the dimension $k$ is two, the morphism maps the curve $E_p$ to the
     plane. Then either the morphism has degree three and the image is a
     line, which leads to the same contradiction, or the morphism is
     an isomorphism and the image is a cubic which has infinitely many
     reflection points, which is also impossible. 
     It remains the case that the dimension $k$ is three, but then $E_p$ has a
     $\mathfrak{g}_3^3$ and is rational, in contradiction to
     \eqref{eq:length4:genus}. 

     This finishes the case that the part of  $B_p$ through $p$ is reduced.

     It remains to consider the case that $B_p$ is not reduced in
     $p$. Using the notation of the proof of Lemma \ref{lem:DD} we write $B_p\equiv
     k\cdot E_p+E'$ with $k\geq 2$, $E_p$ irreducible passing through $p$ and
     $E'$ not containing any component algebraically equivalent to $E_p$. We
     have seen there (see p.\ \pageref{eq:DD:0}) that $E'\not=0$ implies $B_p^2\geq 4$ in
     contradiction to \eqref{eq:length4:1}. We may therefore assume
     $B_p=k\cdot E_p$ with $E_p^2\geq 0$. If $E_p^2\geq 1$, then again
     $B_p^2\geq 4$. Thus $E_p^2=0$. Applying Lemma \ref{lem:curves} to
     $E_p$ we get
     \begin{displaymath}
       3\leq (A+B){\intmult}E_p=A{\intmult}E_p,
     \end{displaymath}
     and hence the contradiction
     \begin{displaymath}
       4\geq A{\intmult}B=k\cdot A{\intmult}E_p\geq 6.
     \end{displaymath}
     This finishes the case that $B_p$ is not reduced in $p$, and
     shows thus that the case $\length(Z'_p)=4$ cannot occur.
   \end{proof}

   \begin{proposition}\label{prop:length3}
     Let $p\in S$ be general and suppose that
     $\length(Z'_p)=3$.  
     Then $B_p$ is an irreducible, smooth, rational curve
     in the pencil $|B|_a$ with $B^2=0$, $A{\intmult}B=3$ and
     $\exists\;s\in H^0\big(B_p,\ko_{B_p}(A_p)\big)$ such that $Z_p'$
     is the zero-locus of $s$.

     In particular, $S\rightarrow |B|_a$
     is a ruled surface and $2B_p$ is a 
     fixed component of $|L-3p|$.
   \end{proposition}
   \begin{proof}
     Since in Section \ref{sec:zero} we have shown that $B$ is nef,
     Lemma \ref{lem:B^2positive} implies 
     \begin{equation}\label{eq:length3:0}
       B^2=0.
     \end{equation}
     Once we have shown that $B_p$ is irreducible and reduced, we then know that
     $|B|_a$ is a pencil and induces a fibration on $S$ whose fibres
     are the elements of $|B|_a$ (see \cite{Kei01} App.\ B.1). In
     particular, the general element of $|B|_a$, which is $B_p$, is
     smooth (see \cite{BHPV04} p.\ 110).
     
     Let us therefore first show that $B_p$ is irreducible and reduced. Since $\kb$
     has no fixed component we know for each irreducible component
     $B_i$ of $B_p=\sum_{i=1}^rB_i$ that $B_i^2\geq 0$, and hence by Lemma
     \ref{lem:curves} that $(A+B){\intmult}B_i\geq
     2$. Thus by \eqref{eq:AB:2} and \eqref{eq:length3:0}
     \begin{displaymath}
       2\cdot r\leq (A+B){\intmult}B=A{\intmult}B+B^2=A{\intmult}B\leq 3,
     \end{displaymath}
     which shows that $B_p$ is irreducible and reduced and that $A{\intmult}B=3$.

     Since $A{\intmult}B=3=\length(Z'_p)$ Proposition \ref{prop:splitting}
     implies that there is a section
     $s_p\in H^0\big(B_p,\ko_{B_p}(A_p)\big)$ such that
     $Z_p'$ is the zero-locus of $s_p$, which is just $3p$.
     Note that for $p\in S$ general and $q\in B_p$ general we have
     $B_p=B_q$ since $|B|_a$ is a pencil, and thus by the construction
     of $B_p$ and $B_q$ we also have
     \begin{displaymath}
       A_p\sim_l L-K-B_p=L-K-B_q\sim_l A_q.
     \end{displaymath}
     But if $A_p$ and $A_q$ are linearly equivalent, then so are the
     divisors $s_p$ and $s_q$ induced on the curve $B_p=B_q$. The
     curve $B_p$ therefore contains a linear series $|\ko_{B_p}(A_p)|$
     of degree three which contains $3q$ for a general point $q\in
     B_p$. In particular, the linear series has no base point and
     induces a morphism $\varphi:B_p\longrightarrow\pp^k$ where $k$ is
     the dimension of the linear series. 

     Suppose that $k$ was one,
     then $\varphi$ would be a morphism of degree three from the curve
     $B_p$ to a line and it would have infinitely many ramification points
     $q$, which is clearly not possible. If $k$ is two, then either
     $\varphi$ has degree three and its image is a line, which leads
     to the same contradiction, or $\varphi$ has degree one and the
     image is a plane cubic. In that case $\varphi$ is a birational
     morphism and either $B_p$ is rational (if $\im(\varphi)$ is
     singular) or $B_p$ is elliptic (if $\im(\varphi)$ is smooth).
     If $B_p$ was an elliptic curve, then 
     the general point $q$ of the cubic $\im(\varphi)$ embedded via
     the $\mathfrak{g}_3^2=|\ko_{B_p}(A_p)|$ would
     be an inflexion point. But that is clearly not possible. Finally,
     if $k$ is three, then $B_p$ has a $\mathfrak{g}_3^3$ and is thus
     rational. Alltogether we have shown that
     \begin{displaymath}
       p_a(B_p)=0,
     \end{displaymath}
     and by the adjunction formula we get
     \begin{equation}
       \label{eq:length3:2}
       K{\intmult}B=2\cdot p_a(B)-2-B^2=-2.
     \end{equation}

     Note also, that $Z'_p\subset B_p$ in view of Table \eqref{eq:3jets}
     implies that $B_p$ and $L_p$ have a common tangent in $p$.
     Suppose that $B_p$ and $L_p$
     have no common component, i.\ e.\ $B_p\not\subset L_p$, then
     \begin{displaymath}
       3\leq\mult_p(B_p)\cdot\mult_p(L_p)<
       L{\intmult}B=A{\intmult}B+B^2+K{\intmult}B=3+K{\intmult}B=1, 
     \end{displaymath}
     which contradicts \eqref{eq:length3:2}. 
     Thus, $B_p$ is at least once contained in $L_p$. 
     Moreover, if
     $2B_p\not\subset L_p$ then by Table \eqref{eq:3jets} $L'_p:=L_p-B_p$
     has multiplicity two in $p$, and it still has a common tangent
     with $B_p$ in $p$, so that
     \begin{equation}\label{eq:length3:3}
       3\leq L'_p{\intmult}B_p=L{\intmult}B-B^2=A{\intmult}B+K{\intmult}B=3+K{\intmult}B=1
     \end{equation}
     again is impossible. We conclude finally, that $B_p$ is at least
     twice contained in $L_p$

     Note finally,
     since $\dim|B|_a=1$ there is a unique curve $B_p$ in $|B|_a$
     which passes through $p$, i.\ e.\ it does not depend on the
     choice of $L_p$, so that in these cases $B_p$ respectively $2B_p$
     is actually a fixed component of $|L-3p|$.
   \end{proof}

%%%%%%%%%%%%%%%%%%%%%%%%%%%%%%%%%%%%%%%%%%%%%%%%%%%%%%%%%%%%%%%%%%%%%%%%%%%%%%%%%%

   \section{Triple-Point Defective Surfaces are Ruled}\label{sec:ruled}

   The considerations of the previous sections prove the following
   theorem. 

   \begin{theorem}[``$S$ is a ruled surface.'']\label{thm:ruled}
     More precisely, let $L$ be a line bundle
     on $S$ such that $L$ is very ample and $L-K$ is ample and base-point-free. Suppose that
     $(L-K)^2>16$ and that for a
     general $p\in S$ the linear system $|L-3p|$ contains a curve
     $L_p$ which has no triple component through $p$, but such that
     $h^1\big(S,\kj_{Z_p}(L)\big)\not=0$ where $Z_p$ is the
     equimultiplicity scheme of $L_p$ at $p$.

     Then there is a
     ruling $\pi:S\rightarrow C$ of $S$ such that $L_p$ contains
     the fibre over $\pi(p)$ with multiplicity two.
   \end{theorem}

   In view of Propositon \ref{prop:h1vanishing} this proves Theorem
   \ref{thm:aim1}.

%    \bibliographystyle{plain}
%    \bibliography{bibliographie}

\end{document}